\documentclass[11pt]{article}
\usepackage[a4paper,margin=1.05in]{geometry}
\usepackage{amsmath,amssymb,amsthm,mathtools}
\usepackage{mathrsfs}
\usepackage{hyperref}
\usepackage[nameinlink]{cleveref}
\usepackage{enumitem}
\usepackage{microtype}
\usepackage{bm}

\def\R{\mathbb{R}}

\newcommand{\ds}{\displaystyle}

\newcommand{\mcl}{\mathcal{L}}

\setlist[itemize]{leftmargin=1.5em}
\setlist[enumerate]{leftmargin=1.5em}

\newtheorem{theorem}{Theorem}[section]

\newtheorem{lemma}[theorem]{Lemma}
\newtheorem{corollary}[theorem]{Corollary}

\theoremstyle{definition}

\newtheorem{example}{Example}[section]
\newtheorem{remarks}{Remarks}[section]

\title{Explicit pulsating fronts and minimal speeds in periodic Fisher-KPP equations}

\author{Lionel Roques\\
\\
\footnotesize{INRAE, BioSP, 84914, Avignon, France}\\
}

\date{}

\begin{document}
\maketitle

\begin{abstract}
We study a Fisher-KPP equation with spatially periodic diffusion and reaction terms. 
We identify a class of periodic media for which the equation admits an explicit, closed-form solution. 
Through a nonlinear change of variables, the problem is reduced to the homogeneous Fisher-KPP equation, allowing us to construct an exact pulsating traveling front that connects the positive periodic stationary state to $0$. 
We also derive an explicit expression for the asymptotic spreading speed and establish new asymptotic and comparison results. Finally, combining our change of variables and eigenvalue transform with existing results on KPP fronts in periodic media, we extend Bramson-type logarithmic delay results to the case of heterogeneous periodic diffusion.
\end{abstract}


\section{Introduction and main assumptions}

In this article, we derive explicit formulas for solutions and spreading speeds of the Fisher-KPP equation with spatially periodic coefficients:
\begin{equation}\label{eq:KPP}
\partial_t u \;=\; \partial_x\!\big(a(x)\,\partial_x u \big) \;+\; f(x,u), 
\qquad  (t,x) \in \R \times \R.
\end{equation}
This class is the natural periodic, spatially heterogeneous counterpart of the homogeneous Fisher-KPP model
$\partial_t u = a\,\partial_{xx}u + f(u)$, introduced independently by Fisher and by Kolmogorov-Petrovskii-Piskunov~\cite{Fis37,KolPet37}.  
Spatial heterogeneity in reaction-diffusion models -- and in particular the periodic framework -- has been extensively used in theoretical ecology, following the influential work of Shigesada and collaborators (see \cite{ShiKaw97,ShiKaw86}), and has also been the subject of a broad mathematical literature establishing numerous qualitative properties;  see the pioneering works~\cite{BerHam02,Wei02,Xin00}.

A central issue in these works concerns the {propagation} properties of solutions of~\eqref{eq:KPP}, including the existence and properties of pulsating traveling fronts and the characterization of asymptotic spreading speeds. While the qualitative theory is well developed, explicit closed-form solutions in heterogeneous periodic media are, to the best of our knowledge, unavailable. In the homogeneous case (constant $a$ and constant linear growth $r=f'(0)$), the minimal wave speed formula $c^*=2\sqrt{r\, a}$ -- which also equals the asymptotic spreading speed -- is classical since the works of Fisher and Kolmogorov-Petrovskii-Piskunov~\cite{Fis37,KolPet37}. However, even in this constant-coefficient case, traveling wave profiles are not explicitly known except in a special case discovered by Ablowitz and Zeppetella~\cite{AblZep79}.

\paragraph{Main assumptions.}
The diffusion coefficient $a(x)$ is assumed to be positive, $L$-periodic for some $L>0$, and of class $\mathcal{C}^{2,\alpha}$ with $\alpha\in(0,1)$. 
The nonlinearity $f:\mathbb{R}\times\mathbb{R}_+\to\mathbb{R}$ is of class $\mathcal{C}^{1,\alpha}$ in $(x,u)$ and $\mathcal{C}^2$ in $u$, is $L$-periodic in $x$, and satisfies $f(x,0)=0$ for all $x\in\mathbb{R}$. We assume the following KPP structure:
\begin{equation}\label{hyp:kpp1}
\forall x\in\mathbb{R},\quad u\mapsto \frac{f(x,u)}{u}\ \text{is nonincreasing on }(0,\infty),
\end{equation}
and
\begin{equation}\label{hyp:kpp2}
\exists M\ge 0\ \text{such that}\ \forall u\ge M,\ \forall x\in\mathbb{R},\quad f(x,u)\le 0.
\end{equation}
For convenience, we define the local growth rate
\[
r(x):=\partial_u f(x,0).
\]
A canonical example of function $f$ satisfying the above assumptions is the logistic form
\begin{equation}\label{eq:logistic}
f(x,u)=r(x)\,u-b(x)\,u^2,\qquad r(x), \ b(x)>0\ \text{are $L$-periodic}.
\end{equation}

\medskip
Define the linear operator
\[
\mathcal{L}_0[a;r]:\ \psi \mapsto (a\psi')' + r(x)\,\psi,
\]
acting on $\mathcal{C}^2$ functions, and denote by $k_0^L[a;r]$ its principal eigenvalue with  $L$-periodic conditions (see \cite{BerHamRoq05a} for existence and properties). 
We further assume
\begin{equation}\label{hyp:lambda1>0}
k_0^L[a;r]>0.
\end{equation}
Under \eqref{hyp:kpp1}-\eqref{hyp:kpp2} and \eqref{hyp:lambda1>0}, there exists a unique positive $L$-periodic stationary state $p(x)$ solving
\[
\partial_x\!\big(a(x)\,\partial_x p \big) \;+\; f(x,p)=0,\qquad x\in\mathbb{R},
\]
see~\cite{BerHamRoq05a}.

\paragraph{Pulsating traveling fronts and asymptotic speed of propagation.}
Following Berestycki-Hamel~\cite{BerHam02}, a pulsating traveling front moving to the right with constant speed $c>0$ is a positive solution of \eqref{eq:KPP} satisfying
\[
u(t+L/c,x)=u(t,x-L), \hbox{ for all } (t,x) \in \R \times \R,
\]
and
\[
\lim_{x \to -\infty} u(t,x)=p(x),\qquad
\lim_{x \to+\infty} u(t,x)=0,
\]
locally in $t$, 
with $p$ the positive $L$-periodic stationary state defined above. 
Under our assumptions, there exists a minimal speed $c^*>0$ such that pulsating traveling fronts exist for all $c\ge c^*$ and do not exist for $0<c<c^*$~\cite{BerHamRoq05b}.
Moreover, the same $c^*$ is the \emph{asymptotic spreading speed} in the sense of Aronson and Weinberger~\cite{AroWei78}, see \cite{BerHamNad05d}:  for any compactly supported nonnegative initial data $u(0,\cdot) \not \equiv 0$, the solution of the Cauchy problem associated with~\eqref{eq:KPP} satisfies
\[
\lim_{t\to+\infty} \sup_{|x|\ge (c^*+\varepsilon)t} u(t,x)=0,
\qquad
\lim_{t\to+\infty} \sup_{|x|\le (c^*-\varepsilon)t} \big|u(t,x)-p(x)\big|=0,
\]
for any fixed $\varepsilon>0$.

The speed $c^*$ can be computed with the Freidlin--G\"artner formula~\cite{GarFre79} (see also~\cite{BerHamNad05d}):
\begin{equation}\label{eq:garfre}
        c^* = \inf_{\lambda > 0} \frac{k_\lambda^L[a;r]}{\lambda}>0,
\end{equation}
where $k_\lambda^L[a;r]$ is the principal eigenvalue of the operator 
\begin{equation*}
   \mcl_\lambda[a ;r ]\, : \, \psi \mapsto  (a \psi')'  - 2 \lambda a  \psi' + \left( \lambda^2 a -  \lambda a' + r \right)\psi,
\end{equation*}
acting on $\mathcal{C}^2$ functions, with $L-$periodic conditions. We recall that the principal eigenvalue $k_\lambda^L$ of $\mcl_\lambda$ is defined as the unique real number for which there exists a $L-$periodic function $\varphi > 0$ satisfying
\[\mathcal{L}_\lambda \varphi = k_\lambda^L\varphi.\]

The existence and uniqueness of the pair $(\varphi,k_\lambda^L)$ (up to multiplication of $\varphi$ by a constant) is guaranteed by the Krein-Rutman theory~\cite{KreRut48}.

\paragraph{Bramson logarithmic delay.}
Assume that $f$ has the logistic form \eqref{eq:logistic}.
The above-mentioned results on the asymptotic spreading speed show that the stationary
state $p(x)$ attracts, locally uniformly, the solutions of~\eqref{eq:KPP} starting from
compactly supported nonnegative initial data $u(0,\cdot)\not\equiv0$.
To study more precisely how the transition from $0$ to $p(x)$ occurs, one can
follow the dynamics of a level set, for instance
\begin{equation}\label{eq:def-Xt}
    X(t):=\max\Big\{x\ge 0 \; : \; u(t,x)=\frac12 \min_{\R} p\Big\}.
\end{equation}
The existence of the asymptotic spreading speed $c^*$ implies that $X(t)/t\to c^*$ as
$t\to +\infty$.
Much more precise descriptions of the position of $X(t)$ have been obtained.
This theory goes back to the pioneering work of Bramson \cite{Bra78,Bra83}
for the homogeneous Fisher--KPP equation, that is when $a\equiv a_0>0$ and
$r\equiv r_0>0$, $b\equiv b_0>0$ are constant. In this case,
\begin{equation}
    \label{eq:X(t)}
    X(t)=c^* t - \frac{3}{2 \lambda^*} \log t + \mathcal O(1)
    \quad\text{as }t\to +\infty,
\end{equation}
with
$c^*=2\sqrt{a_0 r_0},$
$
\lambda^*=\sqrt{r_0/a_0}$.

Since Bramson’s work, there have been many extensions and alternative proofs,
based on PDE techniques rather than the probabilistic arguments of the original
papers; see for instance \cite{HamNol16}.
In particular, \cite{HamNol16} generalized
Bramson’s result to the periodic case considered here, but with a \emph{constant}
diffusion coefficient $a\equiv a_0$.
Their main result implies that the formula~\eqref{eq:X(t)} still holds in this
periodic case, with $c^*$ given by~\eqref{eq:garfre} and $\lambda^*$ the
corresponding minimizer in~\eqref{eq:garfre}.

\paragraph{Contribution.}
Our contributions are fourfold:
\begin{itemize}
  \item[(A)] In Section~\ref{sec:explicit-front}, under the logistic structure~\eqref{eq:logistic}, we identify a class of periodic media $(a,r,b)$ for which the Fisher-KPP equation admits a fully explicit entire solution, obtained via a change of variables that reduces the equation to a standard constant-coefficient KPP form; we then use the classical Ablowitz-Zeppetella profile~\cite{AblZep79} and transform it back to the original variables. This solution is a pulsating traveling front connecting the periodic stationary state $p$ to $0$.
\item[(B)] In Section~\ref{sec:speed}, we derive a general formula that expresses the principal eigenvalue $k_\lambda^L[a;r]$ in terms of the principal eigenvalue of an elliptic operator with constant diffusion. Using this result, we construct a family of periodic Fisher-KPP equations of the form~\eqref{eq:KPP} for which the minimal front speed admits an explicit formula. These results are closely related to those of Nadin~\cite{Nad11}, who analyzed how the spreading speed depends on the coefficients in a space-time periodic Fisher-KPP equation, establishing transformations of the principal eigenvalue that our findings complement.
  \item[(C)] Using the formula of Section~\ref{sec:speed}, we derive explicit comparison bounds (for general $a$ and $r$) and quantitative large-period limits (for general $a$ and constant $r$) for the speed $c^*$. This last result recovers the large-period limit obtained in Corollary~2.4 of~\cite{HamNadRoq11}, but via a simpler argument.
 \item[(D)] In Section~\ref{sec:bramson}, combining our change of variables and eigenvalue transform with the results of~\cite{HamNol16}, we extend their Bramson-type result to the case of a heterogeneous periodic diffusion coefficient $a(x)$. To the best of our knowledge, the Bramson logarithmic delay had previously been
obtained in periodic media only for constant diffusion.

\end{itemize}

\paragraph{Notations.} Define the diffusive coordinate
\begin{equation}\label{eq:Phi}
y= h(x):=\int_0^x a(\xi)^{-1/2}\,d\xi, 
\qquad \Lambda:= h(L)=\int_0^L a(x)^{-1/2}\,dx,
\end{equation}
and
$$\langle a^{-1/2}\rangle:= \frac{1}{L}\int_0^L a(x)^{-1/2}\,dx.$$
Since $a>0$, $h$ is a $\mathcal C^2$ increasing diffeomorphism $\mathbb{R}\to\mathbb{R}$ with $ h(x+L)= h(x)+\Lambda$.

For convenience, we also set:
\begin{equation} \label{def:w}
    w(x):=\tfrac13\, a(x)^{1/4}\, (a(x)^{3/4})''.
\end{equation}
\section{An explicit pulsating traveling front}\label{sec:explicit-front}

In this section, we assume a logistic-like nonlinearity~\eqref{eq:logistic}. 
Fix constants $r_0>0$ and $ b_0>0$, and impose the relations:
\begin{subequations}\label{eq:M-and-B}
\begin{align}
r(x)&= r_0+w(x),\label{eq:M}\\
b(x)&= b_0\,a(x)^{1/4}.
\label{eq:B}
\end{align}
\end{subequations}

\begin{theorem}[Explicit Ablowitz-Zeppetella solution]\label{thm:explicit-front}
Assume~\eqref{eq:logistic} and~\eqref{eq:M-and-B}.  

\noindent (i) The function $\displaystyle p(x):=\frac{r_0}{ b_0}\,a(x)^{-1/4}$ is the unique positive bounded stationary solution of~\eqref{eq:KPP}.

\smallskip
\noindent (ii) Let $\displaystyle c_\mathrm{AZ}:=5\sqrt{ r_0/6}$ and define
\begin{equation}\label{eq:explicit-solution}
u(t,x):=p(x)\,
\Big(1+\exp\big(\sqrt{ r_0/6}\,[ h(x)-c_\mathrm{AZ} t-\xi_0]\big)\Big)^{-2},
\end{equation}
where $ h $ is as in \eqref{eq:Phi} and $\xi_0\in\mathbb{R}$ is arbitrary. 
Then $u$ is a classical entire solution of \eqref{eq:KPP} on $\mathbb{R}\times\mathbb{R}$ satisfying
\[
\lim_{t\to-\infty}u(t,x)=p(x),\qquad
\lim_{t\to+\infty}u(t,x)=0.
\]
\smallskip
\noindent (iii) The solution $u(t,x)$ defined in (ii) is a pulsating traveling front moving to the right with speed $\ds c= \frac{c_{\mathrm{AZ}}}{\langle a^{-1/2}\rangle}$.
\end{theorem}

\begin{example}\label{ex:front}
Let $a(x)=\bigl(1+\varepsilon\cos(2\pi x/L)\bigr)^{2}$ with $|\varepsilon|<1$.
See Fig.~\ref{fig:explicit-entire-solution} for a picture of the corresponding explicit solution at several times.
\end{example}

\begin{figure}[h!]
  \centering
  \includegraphics[width=\linewidth]{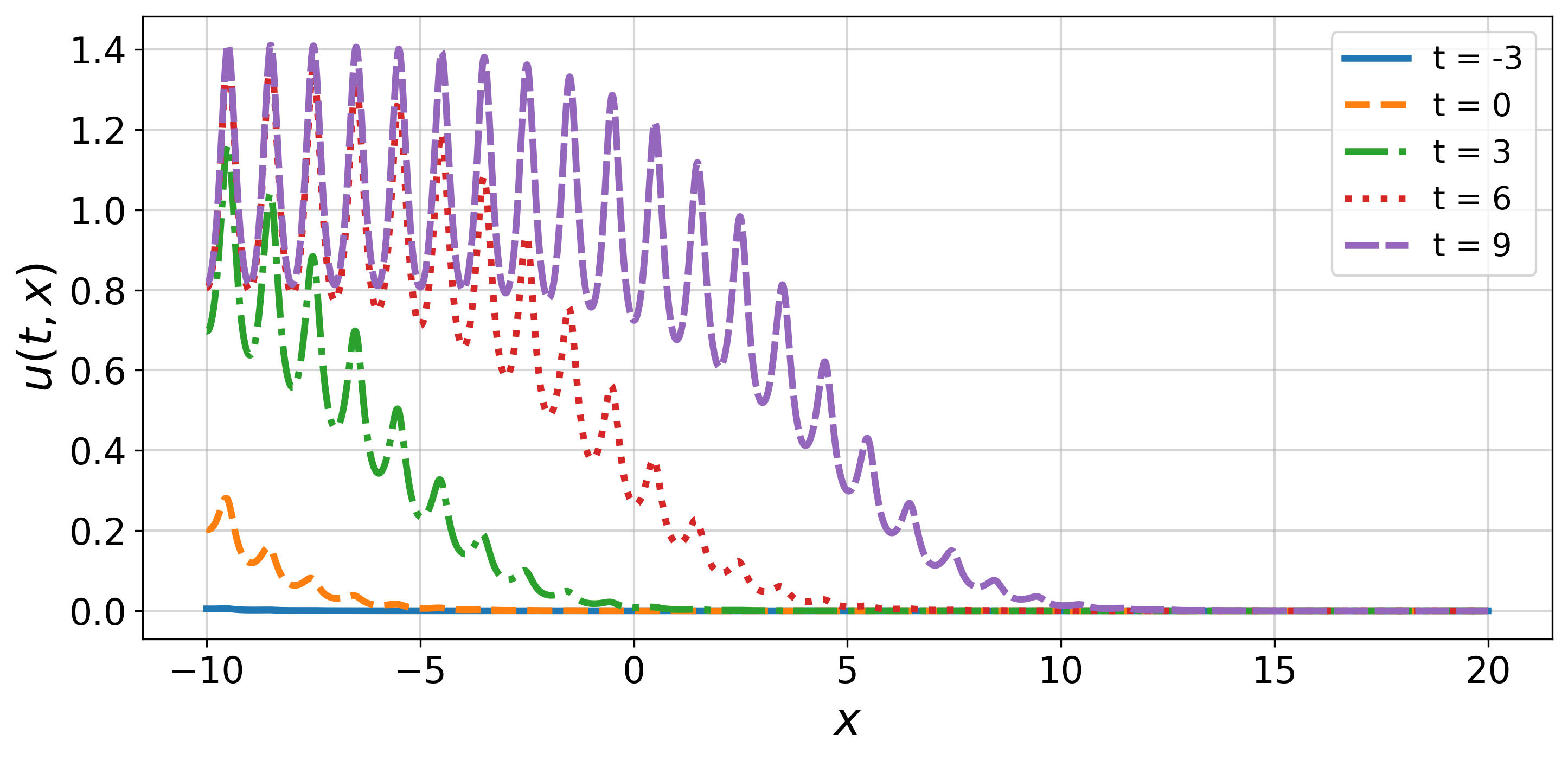}
  \caption{Explicit pulsating traveling front solution $u(t,x)=\dfrac{ r_0}{ b_0}\,a(x)^{-1/4}
  \bigl(1+\exp\bigl(\sqrt{ r_0/6}\,( h(x)-c_{\mathrm{AZ}} t)\bigr)\bigr)^{-2}$ 
  at several times, with $a(x)=(1+\varepsilon\cos(2\pi x/L))^{2}$, $L=1$, $\varepsilon=0.5$, $r_0=1$, $b_0=1$ and
  $c_{\mathrm{AZ}}=5\sqrt{ r_0/6}$. }
  \label{fig:explicit-entire-solution}
\end{figure}

\section{Minimal speed }\label{sec:speed}

We first show that the principal eigenvalue  $k_\lambda^L[a;r]$ can be expressed in terms of the principal eigenvalue of an elliptic operator with constant diffusion.
\begin{theorem}\label{thm:klambda}
We have 
$$k_\lambda^L[a;r]= k_\mu^\Lambda[1 ; R-W],$$with $$\mu=\frac{\lambda}{\langle a^{-1/2}\rangle}, \ R(y):=r(h^{-1}(y)), \ W(y):=w(h^{-1}(y)).$$
\end{theorem}
Note that $k_\mu^\Lambda[1 ; R-W]$ is the principal eigenvalue of the operator $ \mathcal{L}_\mu [1; R-W]$ acting on \emph{$\Lambda$-periodic functions}.

\

As a consequence of Theorem~\ref{thm:klambda}, and using the Freidlin--G\"artner formula \eqref{eq:garfre} we derive a closed form expression for the minimal speed, under a condition on the linear part of $f$:
\begin{equation}\label{eq:M-only}
r(x)= r_0+w(x),
\end{equation}
where $r_0>0$ is constant and $w$ is defined by~\eqref{def:w}.
\begin{corollary}[Explicit minimal KPP speed]\label{thm:cstar}
Assume~\eqref{eq:M-only}. 
Then the KPP minimal speed $c^*>0$ of pulsating traveling fronts for \eqref{eq:KPP} is
\begin{equation}\label{eq:cstar}
c^* \;=\; \frac{2\sqrt{ r_0}}{\displaystyle \langle a^{-1/2}\rangle}.
\end{equation}
\end{corollary}

\begin{example}\label{ex:speed}
With $a(x)=\bigl(1+\varepsilon\cos(2\pi x/L)\bigr)^2$ (with $|\varepsilon|<1$) and \eqref{eq:M-only}, one has
\[
\langle a^{-1/2}\rangle=\frac{1}{L}\!\int_0^L\frac{dx}{1+\varepsilon\cos(2\pi x/L)}
=\frac{1}{\sqrt{1-\varepsilon^2}},
\]
so that
\[
c^* \;=\; 2\sqrt{r_0 \, (1-\varepsilon^2)}.
\]
\end{example}

\

Corollary~\ref{thm:cstar} implies that the explicit solution of Theorem~\ref{thm:explicit-front} is not the pulsating traveling front with minimal speed, as stated in the next corollary. 
\begin{corollary}\label{cor:AZ-vs-cstar}
Under the assumptions of Theorem~\ref{thm:explicit-front} and Corollary~\ref{thm:cstar}, the pulsating traveling front of Theorem~\ref{thm:explicit-front} propagates strictly faster than the minimal KPP speed:
\[
c \;=\; \frac{c_{\mathrm{AZ}}}{\langle a^{-1/2}\rangle}
\;=\; \frac{5}{\sqrt{6}}\,\frac{\sqrt{ r_0}}{\langle a^{-1/2}\rangle}
\;>\;
\frac{2\sqrt{ r_0}}{\langle a^{-1/2}\rangle}
\;=\; c^*.
\]
\end{corollary}
Another consequence of Corollary~\ref{thm:cstar} concerns the propagation of solutions of~\eqref{eq:KPP} in slowly varying media (large $L$), when $r\equiv r_0$ is constant.
\begin{corollary}\label{cor:large-L}
Assume that $a$ and $x\mapsto f(x,\cdot)$ are $1$-periodic, and define the $L$-periodic rescalings $a_L(x):=a(x/L)$ and $f_L(x,\cdot):=f(x/L,\cdot)$. Suppose $r_L(x)=r(x/L)\equiv r_0>0$. Consider the $L$-periodic KPP problem
\begin{equation}\label{eq:KPP_L}
\partial_t u \;=\; \partial_x\!\big(a_L(x)\,\partial_x u \big) \;+\; f_L(x,u), 
\qquad  (t,x) \in \R \times \R,
\end{equation}
and denote by $c^*(L)$ its minimal rightward pulsating front speed. 
Then 
$$c^*(L)=\frac{2\sqrt{r_0}}{\langle a^{-1/2}\rangle} + \mathcal O \left( \frac{1}{L^2}\right) \hbox{ as }L\to +\infty.$$
\end{corollary}

\begin{remarks}
(1) The hypothesis \eqref{hyp:lambda1>0} holds here for all $L$ since $k_0^L[a;r_0]=r_0>0$ (constants are eigenfunctions). 
(2) The large-period limit in Corollary~2.4 of~\cite{HamNadRoq11} was obtained  through a viscosity-solution approach to the periodic eigenvalue problem appearing in the Freidlin--G\"artner formula. Their method analyzes the convergence, as $L\to\infty$, of the associated eigenfunctions to periodic viscosity solutions of a Hamilton-Jacobi equation
from which the limiting speed is deduced. In contrast, the present argument avoids the viscosity-solution framework entirely. It relies only on the parabolic comparison principle applied to problems with slightly perturbed growth rates, together with the explicit speed formula of Corollary~\ref{thm:cstar}. Moreover, the result here provides a slightly more precise quantitative estimate, since it yields an explicit convergence rate rather than only the asymptotic limit of $c^*(L)$.
\end{remarks}

Another consequence is the derivation of new bounds for the minimal speed of pulsating traveling fronts solving~\eqref{eq:KPP} (or equivalently for the asymptotic spreading speed). Note that we only make the KPP assumptions \eqref{hyp:kpp1}-\eqref{hyp:kpp2} and \eqref{hyp:lambda1>0} here.

\begin{corollary}\label{cor:sandwich-interval}
Define
\[
\underline{ r_0}:=\min_{x\in[0,L]}\Big(r(x)-w(x)\Big),
\qquad
\overline{ r_0}:=\max_{x\in[0,L]}\Big(r(x)-w(x)\Big).
\]
The minimal rightward KPP speed satisfies
\[
\frac{2\sqrt{\underline{ r_0}}}{\ \langle a^{-1/2}\rangle\ }
\;\le\; c^* \;\le\;
\frac{2\sqrt{\overline{ r_0}}}{\ \langle a^{-1/2}\rangle\ }.
\]
\end{corollary}

\begin{example}[Constant growth, weakly oscillatory diffusion]\label{ex:const-mu-cos-a}
Let $r(x)\equiv r_0>0$ and choose a $1$-periodic diffusion
\[
a(x)=\bigl(1+\varepsilon\cos(2\pi x)\bigr)^2,\qquad |\varepsilon|<1.
\]
Then $\ds \langle a^{-1/2}\rangle=\frac{1}{\sqrt{1-\varepsilon^2}}.$ Moreover,
\begin{align*}
  w(x) & =\tfrac13\, a(x)^{1/4}\, (a(x)^{3/4})''   \\
  & = -2\pi^2\,\varepsilon\cos(2\pi x)+\mathcal O(\varepsilon^2)\quad\text{as }\varepsilon\to0,
\end{align*}
hence $\|w\|_{L^\infty}\le C\,|\varepsilon|$ for $|\varepsilon|\le\varepsilon_0$ (with $C>0$ independent of $\varepsilon$).
By Corollary~\ref{cor:sandwich-interval}, 
\[
\frac{2\sqrt{r_0-\|w\|_{L^\infty}}}{\langle a^{-1/2}\rangle}
\;\le\; c^*(a,r_0)\;\le\;
\frac{2\sqrt{r_0+\|w\|_{L^\infty}}}{\langle a^{-1/2}\rangle},
\]
and therefore,
\[
c^*(a,r_0)
=2\sqrt{r_0(1-\varepsilon^2)}\;+\;\mathcal O(|\varepsilon|)\quad\text{as }\varepsilon\to0.
\]
\end{example}

\section{Bramson logarithmic delay with periodic diffusion}
\label{sec:bramson}

We show that the main result of~\cite{HamNol16} remains true when the diffusion
coefficient $a(x)$ is not constant.

\begin{theorem}[Theorem 1.2 in~\cite{HamNol16} with periodic diffusion]
\label{cor:bramson-periodic-diffusion}
Assume that $f$ has the logistic form \eqref{eq:logistic}.
Let $u$ be the solution of \eqref{eq:KPP} with compactly supported
initial datum $u_0\ge0$, $u_0\not\equiv0$.
Then there exist a constant $C\ge0$ and a function $\xi:(0,\infty)\to\R$ with
$|\xi(t)|\le C$ for all $t>0$ such that
\begin{equation}\label{eq:bramson-diffusion}
\lim_{t\to\infty}
\left\|
u(t,\cdot)-
U_{c^*}\!\left(
t-\frac{3}{2c^*\lambda^*}\log t+\xi(t),\,\cdot
\right)
\right\|_{L^\infty(0,\infty)}=0,
\end{equation}
where $U_{c^*}$ is the minimal rightward pulsating traveling front for \eqref{eq:KPP},
$c^*$ is given by~\eqref{eq:garfre} and $\lambda^*$ is the corresponding minimizer
in~\eqref{eq:garfre}.
\end{theorem}
If moreover \eqref{eq:M-only} holds, then Corollary~\ref{thm:cstar} yields
$c^*=2\sqrt{r_0}/\langle a^{-1/2}\rangle$, and the explicit computation in the
proof of that corollary shows that the minimizer in \eqref{eq:garfre} is
$\lambda^*=\langle a^{-1/2}\rangle\sqrt{r_0}$.

\section{Proofs}\label{sec:proofs}

We begin with a change of variables that makes the diffusion constant and a rescaling that removes the first derivative produced by the change of spatial variable.

\begin{lemma}\label{lem:change}
Let $a\in \mathcal C^2(\R)$ with $a>0$, let $u:\R\times\R\to\R$ be continuously differentiable in $t$ 
and twice continuously differentiable in $x$  and define $v$ by
\[
u(t,x)=a(x)^{-1/4}\,v\bigl(t, h(x)\bigr) =a(x)^{-1/4}\,v\bigl(t,y\bigr) \hbox{ with }y=h(x).
\]
Then, for any continuous functions $r,b$, and for all $(t,x)\in \R \times \R$
\begin{multline}\label{eq:change-identity}
\partial_t u - \partial_x\!\big(a(x)\,\partial_x u\big) - r(x)u + b(x)u^2
= a(x)^{-1/4}\!\Big(
\partial_t v - \partial_{yy}v - (r(x)-w(x))v  \\
+ b(x)\,a(x)^{-1/4} v^2
\Big),
\end{multline}
where $u=u(t,x)$, $v=v(t,y)=v(t,h(x))$ and $w(x)$ is defined by~\eqref{def:w}.
\end{lemma}

\begin{proof}
Set $u(t,x)=a(x)^{-1/4}v(t, h(x))$. Then $\partial_t u(t,x)=a^{-1/4}(x)\partial_t v(t,h(x))$. For spatial derivatives,
\[
\partial_x u =-\tfrac14 a^{-5/4}a' v+a^{-1/4}\partial_y v\, h '
=-\tfrac14 a^{-5/4}a' v+a^{-1/4}\partial_y v\, a^{-1/2}.
\]
Hence $a\,\partial_x u=-\tfrac14 a' a^{-1/4}v+a^{1/4}\partial_y v$ and
\begin{align*}
\partial_x\!\big(a\,\partial_x u\big)
&=\partial_x\!\Big(-\tfrac14 a' a^{-1/4}v+a^{1/4}\partial_y v\Big) \\
&= -\tfrac14 a'' a^{-1/4}v + \tfrac{1}{16} (a')^2 a^{-5/4} v - \tfrac14 a' a^{-1/4}\partial_y v h'+ \tfrac14 a^{-3/4}a'\partial_y v + a^{1/4} \partial_{yy} v\, h ' \\
& = - a^{-1/4}\tfrac14  \left(  a''   - \frac{(a')^2}{4 \, a}\right) v +  a^{-1/4} \partial_{yy} v \\
& =  a^{-1/4}(\partial_{yy} v -w \, v),
\end{align*}
since $ h '=a^{-1/2}$ and 
$w=\tfrac13\, a^{1/4}\, (a^{3/4})'' =\tfrac14\!\left(a''-\frac{(a')^2}{4a}\right).$
Therefore $\partial_t u-\partial_x(a\,\partial_x u)=a^{-1/4}\big(\partial_t v-\partial_{yy} v+w v\big)$, $r u = r a^{-1/4}v $ and
$b\, u^2=b a^{-1/2}v^2$, yielding \eqref{eq:change-identity}. 
\end{proof}

\begin{proof}[Proof of Theorem~\ref{thm:explicit-front}]
Assume \eqref{eq:M-and-B}. By Lemma~\ref{lem:change}, $u(t,x)$ solves \eqref{eq:KPP} if and only if $v(t,y):= a(x)^{1/4}\, u(t,x)$, with $y= h(x)$ solves the homogeneous Fisher-KPP equation
\begin{equation} \label{eq:homog}
\partial_t v=\partial_{yy} v+ r_0 v- b_0 v^2 \qquad (t,y)\in \R \times\R.
\end{equation}

\noindent (i) The constant stationary state $v\equiv r_0/ b_0$ gives the stationary state
$p(x)=a(x)^{-1/4}\,(r_0/ b_0)$. Uniqueness of the positive stationary state under the KPP assumptions follows from standard results~\cite{BerHamRoq05a}.

\

\noindent (ii) The classical Ablowitz-Zeppetella front \cite{AblZep79}
\[
v(t,y)=\frac{r_0}{ b_0}\Big(1+e^{\sqrt{r_0/6}\,\big(y-c_{\mathrm{AZ}}t-\xi_0\big)}\Big)^{-2},
\qquad c_{\mathrm{AZ}}=5\sqrt{r_0/6},
\]
is an explicit solution of the homogeneous equation~\eqref{eq:homog}; $u(t,x)=a^{-1/4}(x)v(t, h(x))$ gives \eqref{eq:explicit-solution} and the limits $u(t,\cdot)\to p$ as $t\to-\infty$, $u(t,\cdot)\to0$ as $t\to+\infty$.

\

\noindent (iii)
Set $\ds T:=\frac{\Lambda}{c_{\mathrm{AZ}}}$. 
Recall from \eqref{eq:explicit-solution} that
\[
u(t,x)=\frac{ r_0}{ b_0}\,a(x)^{-1/4}
\Big(1+\exp\big(\sqrt{ r_0/6}\,[ h(x)-c_\mathrm{AZ} t-\xi_0]\big)\Big)^{-2}.
\]
Since $a$ is $L-$periodic and $ h(x-L)= h(x)-\Lambda$, we have
\[
u(t,x-L)
=\frac{ r_0}{ b_0}\,a(x)^{-1/4}
\Big(1+\exp\big(\sqrt{ r_0/6}\,[ h(x)-\Lambda-c_\mathrm{AZ} t-\xi_0]\big)\Big)^{-2},
\]
while
\[
u(t+T,x)
=\frac{ r_0}{ b_0}\,a(x)^{-1/4}
\Big(1+\exp\big(\sqrt{ r_0/6}\,[ h(x)-c_\mathrm{AZ} t-c_\mathrm{AZ}T-\xi_0]\big)\Big)^{-2}.
\]
By the choice $T=\Lambda/c_{\mathrm{AZ}}$, 
\begin{equation}\label{eq:pulse-identity-direct}
u(t+T,x)=u(t,x-L)\qquad\text{for all }(t,x)\in\R \times \R.
\end{equation}
Consequently, $u(t,x)$ is a right-moving pulsating traveling front with speed 
$\ds c:=\frac{L}{T}=\frac{c_{\mathrm{AZ}}}{\langle a^{-1/2}\rangle}$.
\end{proof}

\begin{proof}[Proof of Theorem~\ref{thm:klambda}] Let $\psi$ be the principal eigenfunction of the operator $\mcl_\lambda$, i.e., the positive $L-$periodic function such that $\mcl_\lambda \psi = k_\lambda^L[a;r] \psi$, with $\psi(0)=1$ to ensure uniqueness. We note that 
\begin{equation} \label{eq:Lpsi}
    \mcl_\lambda \psi(x) = e^{\lambda \, x} \mcl_0 \left(e^{-\lambda \, x} \psi(x)\right), \ x\in \R.
\end{equation}
For all $y\in \R$, define $x = h^{-1}(y)$ (we recall that $h$ is an increasing diffeomorphism  from $\mathbb{R}$ to $\mathbb{R}$). Define $\varphi(y)$ by the relation:
\begin{equation} \label{def:varphi}
   e^{-\lambda \, x} \psi(x)=a(x)^{-1/4} \, e^{-\mu \, y} \, \varphi(y), \ \ \ \mu:= \frac{\lambda}{\langle a^{-1/2}\rangle}.
\end{equation}
We first observe that $\varphi$ is $\Lambda-$periodic. We have
$$
     e^{-\lambda \, (x+L)} \psi(x+L)=a(x+L)^{-1/4} \, e^{-\mu \, h(x+L)} \, \varphi(h(x+L)).
$$
Since $a$ and $\psi$ are $L-$periodic and since $h(x+L)=h(x) + \Lambda$, we get
$$
     e^{-\lambda \, x} \psi(x)=a(x)^{-1/4} \, e^{\lambda \, L-\mu \, \Lambda} e^{-\mu \, y} \, \varphi(y + \Lambda).
$$
From the definition of $\mu$ and $\Lambda$, $\lambda \, L-\mu \, \Lambda=0$. Thus, 
$$
     e^{-\lambda \, x} \psi(x)=a(x)^{-1/4} \,  e^{-\mu \, y} \, \varphi(y + \Lambda),
$$
and, together with~\eqref{def:varphi}, this shows that $$\varphi(y + \Lambda)= \varphi(y) \hbox{ for all }y\in \R.$$

Next, set $u(x):=e^{-\lambda \, x} \psi(x)$ and $v(y):=e^{-\mu \, y} \, \varphi(y)$. Lemma~\ref{lem:change} with $b\equiv 0$ implies that, for all $x \in \R$, and with $y = h(x)$, 
\begin{equation}
    \mcl_0[a;r]\left( u(x)\right) = a^{-1/4}(x)\!\left( \partial_{yy}v(y) + \Big(r(x)-w(x)\Big)v(y) \right).
\end{equation}
Defining $R(y):= r(h^{-1}(y))$ and $W(y):= w(h^{-1}(y))$, we get:
\begin{equation}
    \mcl_0[a;r]\left( u(x)\right) = a^{-1/4}(x)\mcl_0[1; R-W](v(y)).
\end{equation}
Multiplying this expression by $e^{\lambda \,x}$ and using \eqref{eq:Lpsi}, we get:
\begin{align}
    \mcl_\lambda[a;r]\left( \psi(x)\right)  & = a^{-1/4}(x) e^{\lambda \,x} \mcl_0[1; R-W](e^{-\mu \, y} \, \varphi(y)), \nonumber \\
    & = a^{-1/4}(x) e^{\lambda \,x}  e^{-\mu \, y} \mcl_\mu[1; R-W](\varphi(y)). \label{eq:mcl_mu}
\end{align}
By definition of $k_\lambda^L[a;r]$ and of $\varphi(y)$, we have 
\begin{align*}
\mcl_\lambda[a;r]\left( \psi (x)\right) & = k_\lambda^L[a;r] \psi(x)     \\
& =  k_\lambda^L[a;r]  \,  a(x)^{-1/4} \, e^{\lambda \, x} \, e^{-\mu \, y} \, \varphi(y).
\end{align*}
Substituting into~\eqref{eq:mcl_mu} gives:
$$ \mcl_\mu[1; R-W](\varphi(y)) =  k_\lambda^L[a;r]  \,  \varphi(y),$$and $\varphi$ is positive and $\Lambda$-periodic. Thus, $\varphi$ is the principal eigenfunction of $ \mcl_\mu[1; R-W]$ with $\Lambda$-periodic conditions. Finally, this shows that $k_\lambda^L[a;r] = k_\mu^\Lambda[1;R-W]$.

\end{proof}

\begin{proof}[Proof of Corollary~\ref{thm:cstar}]
Assume \eqref{eq:M-only}: $r(x)= r_0+w(x)$. Setting again $R(y):=r(h^{-1}(y))$ and $W(y):=w(h^{-1}(y))$, we get $R-W=r_0$. Using Theorem~\ref{thm:klambda}, for all $\lambda>0$ and $L>0$, we have $k_\lambda^L[a;r] = k_\mu^\Lambda[1;R-W] = k_\mu^\Lambda[1;r_0]$. Since $ k_\mu^\Lambda[1;r_0]$ is the principal eigenfunction of the operator $\mcl_\mu[1;r_0]$ with $\Lambda-$periodic conditions and since $\mcl_\mu[1;r_0]$ has constant coefficients, the corresponding principal eigenfunction is constant, and $k_\mu^\Lambda[1;r_0]= \mu^2 +r_0$. Thus, with the definition of $\mu=\frac{\lambda}{\langle a^{-1/2}\rangle}$, we get 
$$k_\lambda^L[a;r] = \frac{\lambda^2}{\langle a^{-1/2}\rangle^2} +r_0.$$
By the Freidlin--G\"artner formula,
$$
        c^* = \inf_{\lambda > 0} \frac{k_\lambda^L[a;r]}{\lambda} = \inf_{\lambda > 0} \frac{\lambda}{\langle a^{-1/2}\rangle^2} + \frac{r_0}{\lambda}= \frac{2\sqrt{ r_0}}{\displaystyle \langle a^{-1/2}\rangle}.
$$
\end{proof}

\begin{proof}[Proof of Corollary~\ref{cor:large-L}]
Let $a_L(x):=a(x/L)$ and $r_L\equiv r_0>0$. We set
\[
w_L(x):=\tfrac13\, a_L(x)^{1/4}\, (a_L(x)^{3/4})'' =\tfrac14\!\left(a_L''(x)-\frac{(a_L'(x))^2}{4a_L(x)}\right).
\]
Since $a_L'=\tfrac1L a'(\cdot/L)$ and $a_L''=\tfrac1{L^2} a''(\cdot/L)$, there exists $C>0$ such that
\[
\|w_L\|_{L^\infty(\R)}\ \le\ \frac{C}{L^2}.
\]
Set 
\[
r^{\rm comp}_{L,\pm}(x):=r_0\pm \|w_L\|_\infty+w_L(x).
\]
Then $r^{\rm comp}_{L,-}\le r_0\le r^{\rm comp}_{L,+}$ so the parabolic comparison principle together with Corollary~\ref{thm:cstar} imply
\[
\frac{2\sqrt{r_0-\|w_L\|_\infty}}{\langle a_L^{-1/2}\rangle}
\ \le\ c^*(L)\ \le\
\frac{2\sqrt{r_0+\|w_L\|_\infty}}{\langle a_L^{-1/2}\rangle}.
\]
Noting that 
$$\langle a_L^{-1/2}\rangle = \frac{1}{L}\int_0^L a_L(x)^{-1/2}\,dx =\int_0^1 a(x)^{-1/2}\,dx = \langle a^{-1/2}\rangle,$$
and letting $L\to\infty$ yields
\[
c^*(L)=\frac{2\sqrt{r_0}}{\langle a^{-1/2}\rangle} + \mathcal O \left( \frac{1}{L^2}\right).
\]
\end{proof}

\begin{proof}[Proof of Corollary~\ref{cor:sandwich-interval}]
Define
\[
\underline{ r_0}:=\min_{x\in[0,L]}\Big(r(x)-w(x)\Big),\quad
\overline{ r_0}:=\max_{x\in[0,L]}\Big(r(x)-w(x)\Big).
\]
Then for all $x$,
\[
\underline{r_0}+w(x)\ \le\ r(x)\ \le\ \overline{r_0}+w(x).
\]
The parabolic comparison principle implies that:
\[
c^*\!\left(a,\underline{r_0}+w(x)\right)\ \le\ c^*(a,r)\ \le\
c^*\!\left(a,\overline{r_0}+w(x)\right).
\]
Applying Corollary~\ref{thm:cstar} to each bound gives
\[
\frac{2\sqrt{\underline{ r_0}}}{\langle a^{-1/2}\rangle}
\ \le\ c^*(a,r)\ \le\
\frac{2\sqrt{\overline{ r_0}}}{\langle a^{-1/2}\rangle}.
\]
\end{proof}

\begin{proof}[Proof of Theorem~\ref{cor:bramson-periodic-diffusion}]
With the notations of Lemma~\ref{lem:change}, $u$ solves \eqref{eq:KPP} if and only if $v$ solves
\begin{equation}\label{eq:bramson-v-eq-proof}
\partial_t v=\partial_{yy}v+\tilde f(y,v),
\end{equation}
with
$\tilde f(y,v)
:=\big(r(x)-w(x)\big)v-b(x)a(x)^{-1/4}v^2$, $x=h^{-1}(y)$.
In particular, $\tilde f$ is $\Lambda$-periodic in $y$ and satisfies the KPP and concavity assumptions of~\cite{HamNol16}.

Let $\tilde c^*>0$ the minimal speed  associated with~\eqref{eq:bramson-v-eq-proof}, given by the Freidlin--G\"artner formula:
\[
\tilde c^*=\inf_{\mu>0}\frac{k_\mu^\Lambda[1;R-W]}{\mu},
\]
with the notations of  Theorem~\ref{thm:klambda}, i.e., with 
 $R(y):=r(h^{-1}(y))$, and $W(y):=w(h^{-1}(y))$. Let $\mu^*$ be the corresponding minimizer, 
and let $V_{\tilde c^*}$ be the minimal right-moving
pulsating front for that equation.
By \cite[Theorem~1.2]{HamNol16}, there exist a constant $C\ge0$ and a function
$\tilde\xi:(0,\infty)\to\R$ with $|\tilde\xi(t)|\le C$ for all $t>0$ such that
\begin{equation}\label{eq:HNRR-v-Linfty}
\lim_{t\to\infty}
\Big\|
v(t,y)-
V_{\tilde c^*}\!\left(
t-\frac{3}{2\tilde c^* \mu^*}\log t+\tilde\xi(t),y\right)
\Big\|_{L^\infty(0,\infty)}=0.
\end{equation}

Define
\[
U_{\hat c}(t,x):=a(x)^{-1/4}V_{\tilde c^*}(t,h(x)).
\]
As in the proof of Theorem~\ref{thm:explicit-front}, we observe that $U_{\hat c}$ is a right-moving pulsating traveling
front for \eqref{eq:KPP} with speed 
\[
\hat c=\frac{\tilde c^*}{\langle a^{-1/2}\rangle}.
\]
By Theorem~\ref{thm:klambda},
\[
k_\lambda^L[a;r]=k_\mu^\Lambda[1;R-W]
\quad\text{with}\quad
\mu=\frac{\lambda}{\langle a^{-1/2}\rangle},
\]
so the minimal speed $c^*>0$ for~\eqref{eq:KPP} given by the Freidlin--G\"artner formula \eqref{eq:garfre} satisfies
\[c^*=\inf_{\lambda>0}\frac{k_\lambda^L[a;r]}{\lambda}=\inf_{\mu>0}\frac{k_\mu^\Lambda[1;R-W]}{\langle a^{-1/2}\rangle \mu} =\frac{\tilde c^*}{\langle a^{-1/2}\rangle}= \hat c,\]and
the corresponding minimizers satisfy
\[
\lambda^*=\langle a^{-1/2}\rangle\,\mu^*,
\]
and therefore $c^* \, \lambda^* = \tilde c^* \, \mu^*$.
Finally, evaluating~\eqref{eq:HNRR-v-Linfty} at $y=h(x)$ and multiplying the expression by $a(x)^{-1/4}$, we obtain~\eqref{eq:bramson-diffusion}.

\end{proof}

\section*{Acknowledgements}
This work was supported by the ANR project ReaCh, {ANR-23-CE40-0023-01}.

\bibliographystyle{abbrv}

\end{document}